\documentclass[12pt]{amsart}
\usepackage{amsfonts,amssymb,amscd,xy}
\usepackage[leqno]{amsmath}

\usepackage{amssymb}
\usepackage{amsmath}
\usepackage{amsxtra}
\usepackage{amsthm}

\usepackage[usernames,dvipsnames]{xcolor}
\usepackage{tabularx}
\usepackage{epsfig,amssymb}
\usepackage{bm} 
\usepackage{verbatim} 

\usepackage{hyperref}
\definecolor{labelkey}{rgb}{1,0,0}

\usepackage[OT2,T1]{fontenc} 

\DeclareSymbolFont{cyrletters}{OT2}{wncyr}{m}{n}
\DeclareMathSymbol{\sha}{\mathalpha}{cyrletters}{"58}
\DeclareMathSymbol{\tha}{\mathalpha}{cyrletters}{"51}
\DeclareMathSymbol{\bha}{\mathalpha}{cyrletters}{"42}

\usepackage{verbatim} 

\makeatletter

\DeclareMathAlphabet{\mathcalligra}{T1}{calligra}{m}{n}

\numberwithin{equation}{section}

\newcommand{\lra}{\longrightarrow}

\newcommand{\isoto}{\overset{\!\be\sim}{\to}}

\newcommand{\ksep}{K^{\rm s}}
\def\be{\kern -.1em}
\def\le{\kern 0.03em}
\def\lle{\kern 0.015em}
\def\lbe{\kern -.03em}

\newcommand{\Z}{{\mathbb Z}}
\newcommand{\Q}{{\mathbb Q}}

\newcommand{\spec}{\mathrm{ Spec}\,}

\newcommand{\krn}{\mathrm{Ker}\,}
\newcommand{\img}{\mathrm{Im}\e}
\newcommand{\cok}{\mathrm{Coker}\,}
\newcommand{\Hom}{{\mathrm{Hom}}}
\newcommand{\knr}{K^{\lle\mr{nr}}_{\lbe v}}

\newcommand\e{\kern 0.08em}

\newcommand{\s}{\mathcal }

\newcommand{\mr}{\mathrm }

\newcommand{\pic}{\mathrm{Pic}}
\newcommand{\br}{\mathrm{Br}}

\def\e{\kern 0.1em}

\newcommand{\nr}{\mr{nr}}

\newtheorem{lemma}{Lemma}[section]
\newtheorem{theorem}[lemma]{Theorem}
\newtheorem{proposition-definition}[lemma]{Proposition-Definition}
\newtheorem{corollary}[lemma]{Corollary}
\newtheorem{proposition}[lemma]{Proposition}
\theoremstyle{definition}

\theoremstyle{remark}
\newtheorem{remark}[lemma]{Remark}

\begin{document}

\input xy     
\xyoption{all} 

\title[Brauer groups and N\'eron class groups]{Brauer groups and N\'eron class groups}

\author{Cristian D. Gonz\'alez-Avil\'es}
\address{Departamento de Matem\'aticas, Universidad de La Serena, Cisternas 1200, La Serena 1700000, Chile}
\email{cgonzalez@userena.cl}
\thanks{The author is partially supported by Fondecyt grant 1160004.}
\date{\today}


\dedicatory{A la memoria de mi madre, Gloria Avil\'es Cruz (1937-2019)}

\begin{abstract} Let $K$ be a global field and let $S$ be a finite set of primes of $K$ containing the archimedean primes. We generalize the duality theorem established in \cite{ga12} for the N\'eron $S$-class group of an abelian variety $A$ over $K$ by removing the hypothesis that the Tate-Shafarevich group of $A$ is finite. We also derive an exact sequence that relates the indicated group associated to the Jacobian variety of a proper, smooth and geometrically connected curve $X$ over $K$ to a certain finite subquotient of the Brauer group of $X$.
\end{abstract}

\maketitle

\topmargin -1cm

\smallskip

\section*{0. Introduction}

Let $K$ be a global field, let $S$ be a finite set of primes of $K$ containing the archimedean primes and let $\mathcal O_{K,\e S}$ be the ring of $S$-integers in $K$. For every prime $v$ of $K$, let $K_{\lbe v}$ be the completion of $K$ at $v$. If $v\notin S$, we will write $\mathcal O_{\lbe v}$ for the ring of integers of $K_{\lbe v}$, $k(v)$ for the corresponding residue field and $\knr$ for the maximal unramified extension of $K_{v}$ inside a fixed separable closure of $K_{\lbe v}$. Now let $A$ be a smooth, connected and commutative algebraic group over $K$ which admits a N\'eron model $\s A$ over $\mathcal O_{K,\e S}$. The {\it N\'eron $S$-class group of $A$} is the abelian group
\begin{equation}\label{ncg}
C_{\lbe A\le,\e S}=\cok\!\!\be\left[\lbe A(K)\overset{\!\rho}{\to}\bigoplus_{v\notin S}\Phi_{\lbe v}(\be A\lbe)(k(v))\right]\!,
\end{equation}
where, for every prime $v\notin S$, $\Phi_{\lbe v}(\be A\lbe)$ is the \'etale $k(v)$-group scheme of connected components of $\s A_{\e k(\lbe v)}$ and the $v$-component of the map $\rho$ in \eqref{ncg} is the canonical reduction map $A(K)\to\Phi_{\lbe v}(A)(k(v))$. When $A=\mathbb G_{m,\e K}$ and $S\neq\emptyset$, \eqref{ncg} is canonically isomorphic to the ideal $S$-class group of $K$ \cite[Remark 3.3]{ga12}. Thus \eqref{ncg} may be regarded as a generalization of the ideal $S$-class group of $K$ associated to the pair $(A\le,\lbe K\lle)$. When $A$ is an arbitrary algebraic $K$-torus (i.e., is \'etale-locally isomorphic to $\mathbb G_{\lbe m,\e K}^{r}$ for some $r\geq 1$) and $S\neq\emptyset$, the group \eqref{ncg} was discussed in \cite{ga08} and \cite{ga10}, where the interested reader can find detailed information about \eqref{ncg} in this case. When $A$ is an abelian variety over $K$ with finite Tate-Shafarevich group $\sha^{\le 1}\be(A)$, we defined in \cite{ga12} a canonical perfect pairing of finite abelian groups $C_{\lbe A\le,\e S}\times C^{\le\lle 1}_{\! A^{t}\be,\e S}\to \Q/\Z$, where the group $C^{\le\lle 1}_{\! A^{\lbe t}\be,\e S}$ \eqref{c1p} attached to the abelian variety $A^{t}$ dual to $A$ was expressed (rather inconveniently) as an intersection of certain subgroups of $\prod_{\e v\notin S}\be H^{1}\be(K_{\lbe v},A^{t})$. In this paper we extend the above duality theorem to all abelian varieties $A$ over $K$ (i.e., we dispense with the assumption that $\sha^{\le 1}\be(A)$ is finite) and also obtain a more natural description of $C^{\le\lle 1}_{\! A^{t}\be,\e S}$. Namely, set
\begin{equation}\label{c1pp}
\sha^{1}_{\le\nr}\lbe(S,A^{t}\lle)=\krn\!\!\left[\lle H^{1}\be(K,A^{t}\lle)\to
\prod_{v\notin S}H^{1}\lbe(\knr,A^{t}\e)\times \prod_{v\in S}H^{1}\lbe(K_{v},A^{t}\lle)\right].
\end{equation}
In Section \ref{first} we show that there exists a canonical isomorphism of finite abelian groups $C^{\le\lle 1}_{\! A^{t}\be,\e S}\simeq \sha^{1}_{\le\nr}\lbe(S,A^{t}\lle)/\e\sha^{1}\lbe(\lbe A^{t}\lle)$. Then, in Section \ref{tres}, we define a canonical map $\varphi_{\be A,\e S}\colon \e T\e\sha^{1}\lbe(\lbe A\le)_{\rm div}\to C_{\be A\le,\le S}$ \eqref{fi}, where $T\e\sha^{1}\lbe(\lbe A\le)_{\rm div}$ is the total Tate module of the group $\sha^{\le 1}\lbe(\lbe A\le)_{\rm div}$ of divisible elements of $\sha^{\le 1}\be(A)$, and establish the following generalization of \cite[Theorem 4.9]{ga12}:

\begin{theorem}\label{uno} {\rm{(=\le Theorem \ref{main1}\,)}} Let $A$ be an abelian variety over $K$ and let $A^{t}$ be the abelian variety dual to $A$. Then there exists a canonical perfect pairing of finite abelian groups
\[
C_{\lbe A\le,\e S}/\varphi_{\be\lle A\le,\le S}\le(\e T\e\sha^{\le 1}\lbe(\lbe A\le)_{\rm div}\lbe)\times \sha^{\le 1}_{\le\nr}\lbe(S,A^{t}\lle)/\e\sha^{\le 1}\be(\lbe A^{t}\lle)\to \Q/\Z,
\]
where $C_{\lbe A\le,\e S}$ is the group \eqref{ncg}, $\varphi_{\be A,\e S}$ is the map \eqref{fi} and $\sha^{\le 1}_{\le\nr}\lbe(S,A^{t}\e)$ is the group \eqref{c1pp}.
\end{theorem}

The proof of Theorem \ref{uno} is completely different from, and substantially simpler than, the proof of \cite[Theorem 4.9]{ga12}. Its main ingredient is the generalized Cassels-Tate dual exact sequence established in \cite{gat07}.

\medskip

Very little is known about the group $C_{\lbe A\le,\e S}$ for general abelian varieties $A$, although some information about the (in general, larger) group 
\[
\cok\!\!\be\left[\lbe A(K)_{\rm tors}\to\bigoplus_{v\notin S}\Phi_{\lbe v}(\be A\lbe)(k(v))\right]
\]
can be gleaned from \cite{lor} for certain specific choices of $A$ and $K$. See \cite[Remark 2.28]{lor}. Thus the second main result of this paper, namely Theorem \ref{dos} below, is of interest. In order to state it, we need the following notation.

\smallskip

Let $X$ be a proper, smooth and geometrically connected curve over $K$. If $v\notin S$,
let $\delta^{\le\prime}_{v}$ denote the period of $X_{\be K_{\lbe v}}$, i.e., the least positive degree of a divisor class in $(\le{\rm Pic}\e X_{\lbe K_{\lbe v}^{\rm s}}\be)^{{\rm Gal}\lle(K_{\lbe v}^{\le\rm s}/K_{\lbe v})}$, and let $\delta_{v}^{\le\prime\le\lle\rm nr}$ denote the corresponding quantity associated to $X_{\be K_{v}^{\nr}}$. Then $\delta_{v}^{\le\prime\le\lle\rm nr}$ divides $\delta^{\le\prime}_{v}$ and we set
\begin{equation}\label{dv}
d_{\le v}=\delta^{\le\prime}_{v}/\delta_{v}^{\le\prime\le\lle\rm nr}.
\end{equation}
Now let
\begin{equation}\label{d}
d=\displaystyle{\prod_{v\notin S}}\, d_{\le v}.
\end{equation}
The structural morphism
$X_{\be K_{\lbe v}}\to\spec K_{\lbe v}$ induces a pullback homomorphism between the associated Brauer groups $\br\e K_{\lbe v}\to \br\le X_{\lbe K_{\lbe v}}$. Set
\begin{equation}\label{bra}
\br_{\lbe\textrm{a}}\le  X_{\lbe K_{\lbe v}}=\cok\!\be\left[\e\br\le K_{\lbe v}\to \br\le X_{\be K_{\lbe v}}\lle\right]\lbe.
\end{equation}
Next, the canonical morphism
$X_{\be K_{\lbe v}^{\nr}}\to X_{\be K_{\lbe v}}$ induces a pullback homomorphism
$\br_{\lbe\textrm{a}}\le  X_{\lbe K_{\lbe v}}\to \br_{\lbe\textrm{a}}\le  X_{\lbe K_{\lbe v}^{\nr}}$ and we define\,\footnote{ Note that $\br_{\lbe\textrm{a}}\e  X_{\lbe K_{\lbe v}^{\nr}}=\br\e  X_{\lbe K_{\lbe v}^{\nr}}$ since $\br\e K_{\lbe v}^{\nr}=0$ by \cite[X, \S7, Example (b), p.~162]{sl}.}
\[
\br_{\lbe\textrm{a}}\lbe(X_{\be K_{\lbe v}^{\nr}}\be/\be X_{\lbe K_{\lbe v}}\lbe)=\krn\!\!\left[\le\br_{\lbe\textrm{a}}\e  X_{\lbe K_{\lbe v}}\!\be\to \br_{\lbe\textrm{a}}\e  X_{\lbe K_{\lbe v}^{\nr}}\lbe\right]\!.
\]
It is shown in \cite{bis} that there exists a canonical exact sequence of finite abelian groups
\[
0\to\Hom(\le\br_{\lbe\textrm{a}}\lbe(X_{\be K_{\lbe v}^{\nr}}\be/\be X_{\lbe K_{\lbe v}}\lbe),\Q/\Z\e)\to \Phi_{\be v}(\be J\e)(k(v))\to \Z/\lbe d_{\le v}\le\lle\Z\to 0,
\]
where $J$ is the Jacobian variety of $X$ and $d_{\le v}$ is the integer \eqref{dv}. 
In this paper we establish the following global analog of the above exact sequence.

Let
\begin{equation}\label{brasnr}
\bha_{\le\nr}\lbe(S, X\le)=\krn\!\!\lbe\left[\br_{\lbe\textrm{a}}\e X\to
\prod_{v\notin S}\br_{\lbe\textrm{a}}\lle X_{\be\lle K_{v}^{\nr}}\times \prod_{v\in S}\br_{\lbe\textrm{a}}\lle X_{\be\lle K_{v}}\right],
\end{equation} 
where the $v$-component of the indicated map, for $v\notin S$ (respectively, $v\in S\e$), is the pullback homomorphism $\br_{\lbe\textrm{a}}\e X\to \br_{\lbe\textrm{a}}\lle X_{\be K_{v}^{\nr}}$ (respectively, $\br_{\lbe\textrm{a}}\le X\to \br_{\lbe\textrm{a}}\le X_{\be K_{v}}$) induced by the projection $X_{\be K_{v}^{\nr}}\to X$ (respectively,
$X_{\be K_{v}}\to X\e$). Further, let
\begin{equation}\label{bra-sha}
\bha(\lbe X\lbe)=\krn\!\!\lbe\left[\br_{\lbe\textrm{a}}\le X\to
\prod_{\textrm{all}\, v}\br_{\lbe\textrm{a}}\e X_{\be K_{v}}\right].
\end{equation}
Then the following holds.

\begin{theorem} \label{dos} {\rm{(=\le Theorem \ref{main2}\,)}} Assume that the local periods $\delta_{\lbe v}^{\le\prime}$ of $X$ are pairwise coprime. Then there exists a canonical exact sequence of finite abelian groups
\[
0\to\Hom(\e\bha_{\le\nr}\lbe(S, X\le)\lbe/\e\bha(\lbe X\lbe),\Q/\Z\e)\to C_{\be J,\e S}\le/\varphi_{\lbe J,\e S}(\le T\e\sha^{1}\be(\be J\le)_{\rm div}\be)\to\Z/\lbe d\e\Z\to 0,
\]
where the groups $\bha_{\le\nr}(S, X\le)$ and $\bha(\lbe X\lbe)$ are given by \eqref{brasnr} and \eqref{bra-sha}, respectively, $C_{\be J\le,\e S}$ is the N\'eron $S$-class group of the Jacobian variety $J$ of $X$ \eqref{ncg}, $\varphi_{\be\lle J\le,\le S}$ is the map \eqref{fi} associated to $J$ and $d$ is the integer \eqref{d}.
\end{theorem}

\smallskip

In the last section of the paper we briefly discuss conjectural connections (suggested by the referee) that should exist between the group \eqref{ncg} and the defect of weak approximation on an abelian variety $A$.

\section*{Acknowledgement} I thank the referee for several helpful suggestions that led to an improved presentation of this paper.

\section{Preliminaries}\label{first}

If $B$ is a topological abelian group, its (Pontryagin) dual is the topological abelian group
$B^{\le *}=\text{Hom}_{\e\text{cont.}}(B,\Q/\Z)$ endowed with the
compact-open topology, where $\Q/\Z$ is given the discrete
topology. If $\{B_{i}\colon i\in I\}$
is a collection of discrete and torsion abelian groups, we will make the identification
\begin{equation}\label{ps}
\big(\bigoplus_{\e i\in I} B_{i}\big)^{\! *}=\prod_{\e i\in I}B_{i}^{\le *}
\end{equation}
via the isomorphism of \cite[Lemma 2.9.4(b), p.~64]{rl}. The dual of a morphism of topological abelian groups $f\colon B\to C$ will be denoted by $f^{\le *}\colon C^{\le\lle *}\to B^{\le *}$. If $B$ and $C$ are discrete and torsion, then there exists a canonical isomorphism of profinite abelian groups
\begin{equation}\label{prof}
\cok f^{\le *}\isoto (\e\krn f\e)^{*}.
\end{equation}
Dually, if $B$ and $C$ are profinite, then there exists a canonical isomorphism of discrete and torsion abelian groups
\begin{equation}\label{dis}
\krn f\isoto(\cok f^{\le *})^{*}.
\end{equation}
See \cite[Theorem 23.7, p.~196]{st} and \cite[III, \S2.8, Remark 1, p.~237\e]{bou}.

If $n$ is a positive integer, we will write $B_{\le n}$ for the $n$-torsion subgroup
of $B$, $B/n$ for $B/nB$ and $B_{\lle\rm div}$ for the subgroup of divisible elements of $B$. The {\it total Tate module of $B$} is the group $T B=\varprojlim B_{\le n}$, where the inverse limit is taken over all positive integers $n$ ordered by divisibility. Since $B_{\le n}=\Hom(\Z/n,B)$ and $\varinjlim \Z/n=\varinjlim \frac{1}{n}\Z/\Z=\Q/\Z$, we have
\begin{equation}\label{tb}
T B=\Hom(\Q/\Z,B\lle)=\Hom(\Q/\Z,B_{\lle\rm div})=TB_{\lle\rm div}.
\end{equation}

\begin{lemma}\label{ker-cok}  Let  $A\overset{\!f}{\to}B\overset{\!g}{\to}C$ be  morphisms in an abelian category $\s A$. Then there exists a canonical exact sequence in $\s A$
\[
0\to \krn f\to \krn\lbe(\e g\be\circ\be f\e)\to \krn g\to\cok f\to \cok\lbe(\e g\be\circ\be f\e) \to \cok g\to 0.
\]
\end{lemma}
\begin{proof} See, for example, \cite[1.2]{bey}.
\end{proof}

Recall from the Introduction the global field $K$ and the finite set of primes $S$.
Let $G$ be a commutative $K$-group scheme. If $L$ is a field containing $K$, we define 
\begin{equation}\label{bc}
H^{1}\lbe(L,G\e)=H^{1}\lbe(L,G\!\lbe\times_{\be K}\!\be L\le),
\end{equation}
where the right-hand group is the first fppf cohomology group of the $L$-group scheme $G\!\lbe\times_{\be K}\!\be L$. For every prime $v$ of $K$, the image of a class $\xi\in H^{1}\lbe(K,G\e)$ under the restriction map $H^{1}\lbe(K,G\e)\to  H^{1}\lbe(K_{v},G\e)$ will be denoted by $\xi_{\le v}$. Let $\ksep$ be a fixed separable closure of $K$. For every $v\notin S$, we fix a prime $\overline{v}$ of $\ksep$ lying above $v$ and let $\ksep_{\lbe v}$ be the completion of $\ksep$ at $\overline{v}$. Then $\ksep_{\lbe v}$ is a separable closure of $K_{\lbe v}$ and we will write $\knr$ for the maximal unramified extension of $K_{\lbe v}$ inside $\ksep_{\lbe v}$. Identifying ${\rm Gal}(\ksep_{\lbe v}/\knr)$ with a subgroup of ${\rm Gal}(\ksep_{\lbe v}/\lbe K_{\lbe v})$ in the standard way, we obtain a restriction map $H^{1}\lbe(K_{v},G\e)\to H^{1}\lbe(\knr,G\e)$ and define
\begin{equation}\label{h1-nr}
H^{1}_{\nr}\lbe(K_{v},G\e)=\krn\!\!\left[H^{1}\lbe(K_{v},G\e)\to H^{ 1}\lbe(\knr,G\e)\right].
\end{equation}
By \cite[p.~422]{sh}, the inflation map $H^{1}\lbe({\rm Gal}(\knr\be/\be K_{v}\lbe),G\lbe(\be\knr))\to H^{1}\lbe(K_{v},G\e)$ induces an isomorphism of abelian groups
\begin{equation}\label{h2-nr}
H^{1}\lbe({\rm Gal}(\knr\be/\be K_{v}\lbe),G\lbe(\be\knr))\isoto H^{1}_{\nr}\lbe(K_{v},G\e).
\end{equation}

Next we define
\begin{eqnarray}
\sha^{1}_{\le\nr}\lbe(S,G\e)&=&\krn\!\!\left[\e H^{1}\lbe(K,G\e)\to
\prod_{v\notin S}H^{1}\lbe(\knr,G\e)\times \prod_{v\in S}H^{1}\lbe(K_{v},G\e)\right], \label{sha-nr}\\
\sha_{S}^{1}(G\e)&=&\krn\!\!\left[\le H^{1}\lbe(K,G\e)\to
\prod_{v\in S}H^{ 1}\lbe(K_{v},G\e)\right] \label{sha-s}
\end{eqnarray}
and write
\begin{equation}\label{ls}
\lambda_{\le\lle G, \e S}\colon \sha_{S}^{1}(G\e)\to \prod_{v\notin S}H^{1}\lbe(\lbe K_{v},G\e)
\end{equation}
for the map whose $v$-component, for $v\notin S$, is the restriction to $\sha_{S}^{1}(G\e)\subset H^1(K,G\e)$ of the restriction map $H^{1}\lbe(K,G\e)\to H^{1}\lbe(K_{v},G\e)$. Clearly $\sha^{1}_{\le\nr}\lbe(S,G\e)\subset \sha_{S}^{1}(G\e)$ and $\krn \lambda_{\le\lle G, \e S}=\sha^{1}\be(G\e)$ is the Tate-Shafarevich group of $G$. Now recall the group
\begin{equation}\label{c1p}
C^{\le 1}_{G,\e S}=\left(\e\prod_{v\notin S}H^{1}_{\nr}\lbe(K_{v},G\e)\lbe\right)\cap\img \lambda_{\e G, \e S}\subset \prod_{v\notin S}H^{1}\lbe(K_{v},G\le)
\end{equation}
introduced in \cite[p.~80]{ga12}. Since, for every $v\in S$, the restriction map $H^{1}\lbe(K,G\e)\to
H^{1}\lbe(\knr,G\e)$ factors as $H^{1}\lbe(K,G\e)\to
H^{1}\lbe(K_{v},G\e)\to H^{1}\lbe(\knr,G\e)$, a class $\xi\in \sha_{S}^{1}(G\e)$ lies in 
$\sha^{1}_{\le\nr}\lbe(S,G\e)$ if, and only if, $\xi_{v}\in H^{1}_{\nr}\lbe(K_{v},G\le)$ for every $v\in S$. Consequently, the restriction of $\lambda_{\le\lle G, \e S}$ to $\sha^{1}_{\nr}(S,G\e)$ defines a surjection
$\sha^{1}_{\le\nr}\lbe(S,G\e)\twoheadrightarrow C^{\le 1}_{G,\e S}$ whose kernel is $\sha^{1}\be(G\e)$. Consequently, the following holds

\begin{lemma} \label{shaq=cp1} The map $\lambda_{\le\lle G, \e S}$ \eqref{ls} induces an isomorphism of abelian groups
\[
\sha^{1}_{\le\nr}\lbe(S,G\e)/\e\sha^{1}\be(G\e)\isoto C^{\le 1}_{G,\e S},
\]
where the groups $\sha^{1}_{\le\nr}\lbe(S,G\e)$ and $C^{\le 1}_{\le G,\e S}$ are given by $\eqref{sha-nr}$ and \eqref{c1p}, respectively.\qed
\end{lemma}

\section{The generalized duality theorem}\label{tres}

Let $A$ be an abelian variety over $K$. By \cite[\S 1.4, Theorem 3, p.~19]{blr} and \cite[proof of Corollary 15.3, p.~486]{gw}, $\Phi_{v}(\be A\le)(k(v))$ is finite and equal to zero for all but finitely many primes $v\notin S$. Thus the group \eqref{ncg} is finite. Further, since $\prod_{v\notin S}\!\Phi_{v}(\be A\le)(k(v))=\bigoplus_{v\notin S}\!\Phi_{v}(\be A\le)(k(v))$, we may rewrite \eqref{ncg} as
\begin{equation}\label{ncgp}
C_{\lbe A\le,\e S}=\cok\!\!\be\left[\lbe A(K)\overset{\!\rho}{\to}\prod_{v\notin S}\Phi_{\lbe v}(\be A\lbe)(k(v))\right]\!.
\end{equation}

\begin{lemma}\label{gd} For every $v\notin S$, there exists a canonical isomorphism of finite abelian groups
\[
\Phi_{v}(\be A\le)(k(v))\isoto H^{1}\lbe({\rm Gal}(\knr\be/\be K_{v}\lbe),A^{t}\be(\be\knr))^{*}.
\]
\end{lemma}
\begin{proof} By \cite[Theorem 4.3 and its proof]{mcc}, Grothendieck's pairing
\[
\Phi_{v}\lbe(\be A\le)\lbe(\le k(v)^{\rm s})\times \Phi_{v}\lbe(\be A^{t}\le)\lbe(\le k(v)^{\rm s})\to\Q/\Z
\]
induces a perfect pairing of finite groups
\begin{equation}\label{igp}
(-,-)_{v}\colon\Phi_{v}(\lbe A\le)(k(v))\times H^{1}\lbe(k(v),\Phi_{v}(\lbe A^{t}\le)\lle)\to \Q/\Z.
\end{equation}
Now, by \cite[proof of Proposition I.3.8, p.~47]{adt}, the reduction map $A^{t}\lbe(\knr)\to\Phi_{v}\lbe(\be A^{t}\le)\lbe(\le k(v)^{\rm s}\le)$ and the canonical isomorphism ${\rm Gal}(\knr\be/\lbe K_{v})\simeq
{\rm Gal}(k(v)^{\rm s}\be/\lbe k(v))$ induce an isomorphism of finite groups
\begin{equation}\label{psiv}
\psi_{\le v}\colon H^{1}\lbe({\rm Gal}(\knr\be/\lbe K_{v}),A^{t}(\lbe\knr))\isoto H^{1}\lbe(k(v),\Phi_{v}(\lbe A^{t}\le)\lle).
\end{equation}
The isomorphism of the lemma is the composition
\[
\Phi_{v}(\lbe A\le)(k(v))\isoto H^{1}\lbe(k(v),\Phi_{v}(\lbe A^{t}\le)\lle)^{*}\underset{\!\psi_{v}^{*}}{\isoto}
H^{1}\lbe({\rm Gal}(\knr\be/\be K_{v}\lbe),A^{t}\lbe(\be\knr))^{*},
\]
where the first isomorphism is induced by the pairing \eqref{igp}.
\end{proof}

\begin{remark}\label{fin} By the lemma applied to $A^{t}$, the comments before it and \eqref{h2-nr}, the isomorphic groups $H^{1}\lbe({\rm Gal}(\knr\be/\be K_{v}\lbe),A(\be\knr))$ and $H^{1}_{\nr}\lbe(K_{v},A\e)$ are finite and equal to zero for all but finitely many primes $v\notin S$.
\end{remark}

Now set
\begin{equation}\label{qsha}
\tha^{1}\be(A\e)=\cok\!\!\left[\le H^{1}\lbe(K,A\e)\to
\bigoplus_{\text{ all }v}H^{ 1}\lbe(K_{v},A\e)\right]
\end{equation}
and let
\begin{equation}\label{ggs}
\beta_{\lbe A,\e S}\colon \bigoplus_{v\notin S}H^{1}_{\nr}\lbe(K_{v},A\le)\to \tha^{1}\be(A\e)
\end{equation}
be the restriction of the canonical projection $\bigoplus_{\e\text{all }v}\be H^{1}\lbe(K_{v},A\e)\twoheadrightarrow \tha^{1}\be(A\e)$ to the finite subgroup $\bigoplus_{\le v\notin S}\be H^{1}_{\nr}\lbe(K_{v},A\le)$ of $\bigoplus_{\e\text{all }v}\be H^{1}\lbe(K_{v},A\e)$ (see Remark \ref{fin}).  An application of Lemma \ref{ker-cok} to the pair of maps
\[
H^{1}\lbe(K,A\e)\to \bigoplus_{\text{all }v}H^{1}\lbe(K_{v},A\e)\to
\bigoplus_{v\notin S}H^{1}\lbe(\knr,A\e)\times \bigoplus_{v\in S}H^{1}\lbe(K_{v},A\e)
\]
yields an exact sequence of abelian groups
\begin{equation}\label{ex}
0\to \sha^{1}\lbe(\lbe A\le)\to \sha^{1}_{\le\nr}\lbe(S,A\le)\to \displaystyle{\bigoplus_{v\notin S}}\e H^{1}_{\nr}\lbe(K_{v},A\e)\overset{\!\beta_{\lbe A,\lle S}}{\lra} \tha^{1}\be(A\e).
\end{equation}
Thus the following holds

\begin{proposition}\label{exca} There exists a canonical isomorphism of finite abelian groups
\[
\sha^{1}_{\le\nr}\lbe(S,A\e)/\e\sha^{1}\lbe(A\e)\isoto\krn\!\!\lbe\left[\,\bigoplus_{v\notin S}H^{1}_{\nr}\lbe(K_{v},A\e)\overset{\!\beta_{\lbe A\lbe,\le S}}{\lra} \tha^{1}\lbe(A\e)\right],
\]
where $\sha^{1}_{\le\nr}\lbe(S,A\e)$ is the group \eqref{sha-nr} and $\beta_{\le A,\le S}$ is the map \eqref{ggs}.
\end{proposition}

Now we let
\begin{equation}\label{ex0}
\cok \beta_{\lbe A\lle,\e S}^{\le *}\isoto (\e\sha^{1}_{\le\nr}\lbe(S,A\e)/\e\sha^{1}\lbe(A\e))^{*},
\end{equation}
be the composition of canonical isomorphisms of finite abelian groups
\[
\cok \beta_{\lbe A\lle,\e S}^{\le *}\underset{\eqref{prof}}\isoto (\e\krn\beta_{\lbe A,S}\e)^{*}\isoto (\e\sha^{1}_{\le\nr}\lbe(S,A\e)/\e\sha^{1}\lbe(A\e))^{*},
\]
where the second map is the dual of the isomorphism in Proposition \ref{exca}\,.
Next let
\begin{equation}\label{ggs2}
\widetilde{\beta}_{\lbe A\lle,\e S}\colon \bigoplus_{v\notin S}H^{1}\lbe({\rm Gal}(\knr\be/\be K_{v}\lbe),A\lbe(\be\knr))\to \tha^{1}\be(A\e)
\end{equation}
be the composition
\[
\bigoplus_{v\notin S}H^{1}\lbe({\rm Gal}(\knr\be/\be K_{v}\lbe),A\lbe(\be\knr))\isoto\bigoplus_{v\notin S}H^{1}_{\nr}\lbe(K_{v},A\le)\overset{\!\beta_{\lbe A,S}}{\lra} \tha^{1}\be(A\e),
\]
where the first map is the direct sum over $v\notin S$ of the isomorphisms \eqref{h2-nr} and $\beta_{\lbe A,S}$ is the map \eqref{ggs}. The dual of \eqref{ggs2} is a map
\begin{equation}\label{ggs3}
\widetilde{\beta}_{\lbe A,\e S}^{\e *}\colon \tha^{1}\lbe(A\e)^{*}\to\prod_{v\notin S}H^{1}\lbe({\rm Gal}(\knr\be/\be K_{v}\lbe),A\lbe(\be\knr))^{*}.
\end{equation}
We define an isomorphism of finite abelian groups
\begin{equation}\label{bd}
\cok \widetilde{\beta}_{\be A\lle,\le S}^{\le *}\isoto (\e\sha^{1}_{\le\nr}\lbe(S,A\e)/\e\sha^{1}\lbe(A\e))^{*}
\end{equation}
by the commutativity of the diagram
\[
\xymatrix{\cok \beta_{\lbe A\lle,\e S}^{\le *}\ar[d]_(.45){\sim}\ar[r]^(.35){\eqref{ex0}}_(.35){\sim}&(\e\sha^{1}_{\le\nr}\lbe(S,A\e)/\e\sha^{1}\lbe(A\e))^{*}\\
\cok \widetilde{\beta}_{\lbe A\lle,\e S}^{\le *}\ar[ur],&
}
\]
where the vertical map is induced by the product over $v\notin S$ of the duals of the isomorphisms \eqref{h2-nr}. The isomorphism \eqref{bd} will play an auxiliary role in the proof of Theorem \ref{uno}\,.

\medskip

For every positive integer $n$, the exact sequence of fppf sheaves on $\spec K$
\begin{equation}\label{nse}
0\to A_{\e n}\to A\overset{\!n}{\to}A\to 0
\end{equation}
induces an exact sequence of abelian groups
\begin{equation}\label{sse}
0\to A(K)/n\to H^{1}(K, A_{\e n})\to H^{1}(K, A\lle)_{\le n}\to 0.
\end{equation}
Now an application of Lemma \ref{ker-cok} to the pair of maps
\[
H^{1}(K, A_{\e n})\twoheadrightarrow H^{1}(K, A\lle)_{\le n}\to \prod_{\e\text{all }v}H^{1}(K_{v}, A\lle),
\]
using the exactness of \eqref{sse}, yields an exact sequence of projective systems of abelian groups with canonical transition maps
\begin{equation}\label{ssf}
0\to (A(K)/n)\to ({\rm Sel}_{\le n}(\lbe A\le))\to (\e\sha^{1}\be(\lbe A\le)_{n})\to 0,
\end{equation}
where 
\begin{equation}\label{tsel1}
{\rm Sel}_{\le n}(\lbe A\le)=\krn\!\!\left[ H^{1}(K,A_{\le n})\to\prod_{\text{all } v}H^{1}(K_{v},A\lle)\right]
\end{equation}
is the $n$-th Selmer group of $A$. Since $(A(K)/n)$ has surjective transition maps, an application of \cite[Proposition 10.2, p.~102]{am} to \eqref{ssf} using \eqref{tb} yields an exact sequence of profinite abelian groups 
\begin{equation}\label{seq0}
0\to A\lbe (K\le)^{\wedge}\to T\e{\rm Sel}(\lbe A)\to T\e\sha^{1}\be(\lbe A\le)_{\rm div}\to 0,
\end{equation}
where
\begin{equation}\label{at}
A\lbe (K\le)^{\wedge}=\varprojlim_{n} A\lbe(K)/n
\end{equation}
is the adic (equivalently, profinite \cite[\S3]{se}, \cite{mic}) completion of $A\lbe (K\le)$ and 
\begin{equation}\label{tsel2}
 T\e{\rm Sel}(\lbe A\le)=\varprojlim_{n}{\rm Sel}_{\le n}(\lbe A\le)
\end{equation}
is the pro-Selmer group of $A$. The first nontrivial map in the sequence \eqref{seq0} is the projective limit over $n$ of the canonical maps $A(K)/n\to {\rm Sel}_{\le n}(\lbe A\le)$ induced by the connecting maps in fppf cohomology $A(K)\to H^{1}(K, A_{\e n})$ associated to the sequence \eqref{nse}. We will identify $A\lbe (K\le)^{\wedge}$ with its image in $T\e{\rm Sel}(\lbe A)$ under the indicated map. Thus \eqref{seq0} induces an isomorphism of profinite abelian groups
\begin{equation}\label{iso1}
T\e\sha^{1}\be(\lbe A\le)_{\rm div}\isoto T\e{\rm Sel}(\lbe A)/A\lbe (K\le)^{\wedge}.
\end{equation}

For every archimedean prime $v$ of $K$, let $H^{\le 0}\lbe(K_{\lbe v}, A\le)$ be the group of connected components of $A(K_{\lbe v})$. Note that, since the identity component of $A(K_{v})$ is a divisible abelian group \cite[Remark 3.7, p.~46]{adt}, we have $H^{\le 0}\lbe(K_{\lbe v}, A\le)/n=A(K_{v})/n$ for every positive integer $n$.

An application of Lemma \ref{ker-cok} to the pair of maps
\[
H^{1}(K, A_{\e n})\to\prod_{\e\text{all }v}H^{1}(K_{v}, A_{\e n}\lle)\to \prod_{\e\text{all }v}H^{1}(K_{v}, A\lle)
\]
yields a canonical (localization) map ${\rm Sel}_{\le n}(\lbe A\le)\to \prod_{\e\text{all }v}\be\krn\be[H^{1}(K_{v}, A_{\e n}\lle)\be\to\! H^{1}(K_{v}, A\lle)]$. Now we consider 
the composition
\begin{equation}\label{new}
{\rm Sel}_{\le n}(\lbe A\le)\to \prod_{\e\text{all }v}\krn\be[H^{1}(K_{v}, A_{\e n}\lle)\be\to\! H^{1}(K_{v}, A\lle)]\isoto \prod_{\e\text{all }v}\be H^{\le 0}\lbe(K_{\lbe v}, A\le)/n\,,
\end{equation}
where the $v$-component of the second map is the inverse of the isomorphism
\[
H^{\le 0}\lbe(K_{\lbe v}, A\le)/n=A(K_{v})/n\isoto \krn\be[H^{1}(K_{v}, A_{\e n}\lle)\to H^{1}(K_{v}, A\lle)]
\]
induced by the connecting map in fppf cohomology $A(K_{v})\to H^{1}(K_{v}, A_{\e n})$ associated to the sequence \eqref{nse} over $\spec K_{v}$. By \cite[Remark 3.4]{gat12}, the projective limit over $n$ of the maps \eqref{new} is an {\it injection}
\begin{equation}\label{net}
T\e{\rm Sel}(\lbe A)\hookrightarrow \prod_{\e\text{all }v}\be H^{\le 0}\lbe(K_{\lbe v}, A\le)^{\wedge}=\prod_{\e\text{all }v}\be H^{\le 0}\lbe(K_{\lbe v}, A\le),
\end{equation}
where the equality follows from the fact that $H^{\le 0}(K_{\lbe v}, A\le)$ is profinite for every prime $v$ of $K$. Now we consider the composition
\begin{equation}\label{nes}
T\e{\rm Sel}(\lbe A)\hookrightarrow \prod_{\e\text{all }v}\be H^{\le 0}\lbe(K_{\lbe v}, A\le)\isoto \prod_{\e\text{all }v}\be H^{\le 1}(K_{\lbe v}, A^{t}\le)^{*},
\end{equation}
where the first arrow is the map \eqref{net} and the $v$-component of the second arrow is the Tate local duality isomorphism $H^{\le 0}\lbe(K_{\lbe v}, A\le)\isoto H^{\le 1}(K_{\lbe v}, A^{t}\le)^{*}$. It is shown in \cite{gat07} that \eqref{nes} identifies $T\e{\rm Sel}(\lbe A)$ with the kernel of the canonical map $\prod_{\e\text{all}\e v}\be H^{\le 1}(K_{\lbe v}, A^{t}\le)^{*}\to H^{\le 1}(K, A^{t}\le)^{*}$ which, via an application of \eqref{ps}, can be identified with the dual of the canonical localization map $H^{\le 1}(K, A^{t}\le)\to\bigoplus_{\e\text{all}\e v}\be H^{\le 1}(K_{\lbe v}, A^{t}\le)$. Thus, if $\tha^{1}\be(\lbe A^{t}\le)$ is the group \eqref{qsha} associated to $A^{t}$, then there exists a canonical isomorphism of profinite abelian groups
\begin{equation}\label{nes2}
T\e{\rm Sel}(\lbe A)\isoto \tha^{1}\be(\lbe A^{t}\le)^{*},
\end{equation}
namely the composition
\[
T\e{\rm Sel}(\lbe A)\underset{\eqref{nes}}\isoto \krn\!\!\left[\,\prod_{\e\text{all}\e v}\be H^{\le 1}(K_{\lbe v}, A^{t}\le)^{*}\to H^{\le 1}(K, A^{t}\le)^{*}\right]\underset{\eqref{dis}}\isoto \tha^{1}\be(\lbe A^{t}\le)^{*}.
\]
Now let
\begin{equation}\label{mps1}
\widetilde{\gamma}_{\lbe A,\e S}\colon T\e{\rm Sel}(\be A\le)\to\prod_{v\notin S}H^{1}\lbe({\rm Gal}(\knr\be/\be K_{v}\lbe),A^{t}\lbe(\be\knr))^{*}
\end{equation}
be the composition
\[
T\e{\rm Sel}(\be A\le)\underset{\eqref{nes2}}\isoto\tha^{1}\be(\lbe A^{t}\le)^{*}\overset{\!\!\widetilde{\beta}_{\!\be A^{\lbe t}\be,\le\lle S}^{\le*}}{\lra}\prod_{v\notin S}H^{1}\lbe({\rm Gal}(\knr\be/\be K_{v}\lbe),A^{t}\lbe(\be\knr))^{*},
\]
where $\widetilde{\beta}_{\! A^{\lbe t}\be,\le\lle S}^{\le*}$ is the map \eqref{ggs3} associated to $A^{\lbe t}$. Then there exists a canonical isomorphism of finite abelian groups
\begin{equation}\label{can}
\cok \widetilde{\gamma}_{\lbe A,\e S}\isoto(\e\sha^{1}_{\le\nr}\lbe(S,A^{\lbe t}\e)/\e\sha^{1}\lbe(\lbe A^{t}))^{*},
\end{equation}
namely the composition
\[
\cok \widetilde{\gamma}_{\lbe A,\e S}\isoto \cok\widetilde{\beta}_{\! A^{\lbe t}\be,\le\lle S}^{\le*}\isoto (\e\sha^{1}_{\le\nr}\lbe(S,A^{\lbe t}\e)/\e\sha^{1}\lbe(\lbe A^{t}))^{*},
\]
where the first map is induced by the identity map on 
$\prod_{v\notin S}H^{1}\lbe({\rm Gal}(\knr\be/\be K_{v}\lbe),A^{t}\lbe(\be\knr))^{*}$ and the second map is the isomorphism \eqref{bd} associated to $A^{\lbe t}$.

\smallskip

Now let
\begin{equation}\label{ggsd}
\gamma_{\lbe A,\e S}\colon T\e{\rm Sel}(\be A\le)\to\prod_{v\notin S}\Phi_{v}(\be A\le)(k(v))
\end{equation}
be defined by the commutativity of the diagram
\begin{equation}\label{dio}
\xymatrix{&&\displaystyle{\prod_{v\notin S}}\,\Phi_{v}(\be A\le)(k(v))\ar[d]^(.55){\sim}\\
T\e{\rm Sel}(\be A\le)\ar[urr]^(.4){\gamma_{\lbe A,\e S}}\ar[rr]^(.35){\widetilde{\gamma}_{\lbe A,S}}&&\displaystyle{\prod_{v\notin S}}H^{1}\lbe({\rm Gal}(\knr\be/\be K_{v}\lbe),A^{t}\lbe(\be\knr))^{*},
}
\end{equation}
where $\widetilde{\gamma}_{\lbe A,\le S}$ is the map \eqref{mps1} and the $v$-component of the vertical arrow is the isomorphism of Lemma {\rm \ref{gd}}\,. Then there exists a canonical isomorphism of finite abelian groups
\begin{equation}\label{can2}
\cok\gamma_{\lbe A,\e S}\isoto(\e\sha^{1}_{\le\nr}\lbe(S,A^{\lbe t}\e)/\e\sha^{1}\lbe(\lbe A^{t}))^{*},
\end{equation}
namely the composition
\[
\cok\gamma_{\lbe A,\e S}\isoto\cok \widetilde{\gamma}_{\lbe A,\e S}\underset{\eqref{can}}{\isoto} (\e\sha^{1}_{\le\nr}\lbe(S,A^{\lbe t}\e)/\e\sha^{1}\lbe(\lbe A^{t}))^{*},
\]
where the first map is induced by the vertical map in \eqref{dio}.

\begin{lemma}\label{comm} The following diagram commutes 
\[
\xymatrix{A\lbe (K\le)^{\wedge}\ar@{^{(}->}[d]\ar[rr]^(.35){\widehat{\rho}}&&\displaystyle{\prod_{v\notin S}}\,\Phi_{v}(\be A\le)(k(v))\,,\\
T\e{\rm Sel}(\be A\le)\ar[urr]_(.4){\gamma_{\lbe A,\e S}}&&
}
\]
where the vertical map is the first nontrivial map in the sequence \eqref{seq0}, 
$\gamma_{\lbe A,\e S}$ is the map \eqref{ggsd} and the horizontal map is the completion of the reduction map $\rho\colon A\lbe (K\le)\to \prod_{v\notin S}\!\Phi_{v}(\be A\le)(k(v))$.
\end{lemma}
\begin{proof} It suffices to check the commutativity of the outer contour of the diagram
\[
\xymatrix{A\lbe (K\le)^{\wedge}\ar@{^{(}->}[d]\ar[rr]^(.35){\widehat{\rho}}&&\displaystyle{\prod_{v\notin S}}\,\Phi_{v}(\be A\le)(k(v))\ar[d]^{\sim}\\
T\e{\rm Sel}(\be A\le)\ar[urr]^(.4){\gamma_{\lbe A,\e S}}\ar[rr]^(.35){\widetilde{\gamma}_{\lbe A,S}}&&\displaystyle{\prod_{v\notin S}}H^{1}\lbe({\rm Gal}(\knr\be/\be K_{v}\lbe),A^{t}\lbe(\be\knr))^{*},
}
\]
where $\widetilde{\gamma}_{\lbe A,\le S}$ is the map \eqref{mps1} and the $v$-component of the right-hand vertical map is the isomorphism of Lemma {\rm \ref{gd}}\,. The claimed commutativity follows from the definitions of the indicated maps using the known compatibility \cite[p.~1112]{mcc} of the pairing \eqref{igp} (induced by Grothendieck's pairing) with Tate's local duality pairing $\langle -,- \rangle_{v}\colon A(K_{v})\times H^{1}\lbe(K_{v},A^{t}\le)\to \Q/\Z$ for $v\notin S$, i.e., the formula
\[
(\e\rho_{\le v }(\lbe P\le), \psi_{\le v}(\xi))_{v}=\langle P,{\rm inf}_{v}(\xi)\rangle_{v}
\]
for $P\in A(K_{v})$ and $\xi\in H^{1}\lbe({\rm Gal}(\knr\be/\be K_{v}\lbe),A^{t}\lbe(\be\knr))$, where $\rho_{\le v }\colon A(K_{v})\to \Phi_{v}(\be A\le)(k(v))$ is the reduction map, $\psi_{\le v}$ is the map \eqref{psiv} and ${\rm inf}_{v}\colon H^{1}\lbe({\rm Gal}(\knr\be/\be K_{v}\lbe),A^{t}\lbe(\be\knr))\to H^{1}\lbe(K_{v},A^{t}\le)$ is the inflation map.
\end{proof}

\smallskip

{\it We may now prove Theorem} \ref{uno}.

\medskip

By Lemma \ref{comm}\,, the left-hand square below commutes
\begin{equation}\label{gdi}
\xymatrix{A\lbe (K\le)^{\wedge}\ar@{^{(}->}[r]\ar@{=}[d]&T\e{\rm Sel}(\be A\le)
\ar[r]\ar[d]^{\gamma_{\lbe A,\le S}}& T\e{\rm Sel}(\be A\le)/A\lbe (K\le)^{\wedge}\ar@{-->}[d]^{\overline{\gamma}_{\be\le A,\le S}}\ar[r]& 0\\
A\lbe (K\le)^{\wedge}\ar[r]^(.35){\widehat{\rho}}&\displaystyle{\prod_{v\notin S}}\Phi_{v}(\be A\le)(k(v))
\ar[r]^(.55){\pi}&\cok\e\widehat{\rho}\ar[r]&0,
}
\end{equation}
which establishes the existence of a unique map $\overline{\gamma}_{\be\le  A,\le S}\colon T\e{\rm Sel}(\be A\le)/A\lbe (K\le)^{\wedge}\to \cok\e\widehat{\rho}$ such that the full diagram \eqref{gdi} commutes. Note that the canonical projection map $\pi$ in \eqref{gdi} induces an isomorphism of finite abelian groups
\begin{equation}\label{gdi2}
\cok\gamma_{\lbe  A,\le S}\isoto \cok\e\overline{\gamma}_{\lbe  A,\le S}.
\end{equation}
Now, by \eqref{ncgp} and \cite[Proposition 3.2.5, p.~87]{rl}, there exists a canonical isomorphism of finite abelian groups $C_{\be A\le,\e S}\isoto\cok\e\widehat{\rho}$\,. Let
\begin{equation}\label{fip}
\overline{\gamma}_{\be\le  A,S}^{\e\prime}\colon T\e{\rm Sel}(\be A\le)/A\lbe (K\le)^{\wedge}\to C_{\be A\le,\e S}
\end{equation}
be defined by the commutativity of the diagram
\begin{equation}\label{fip2}
\xymatrix{T\e{\rm Sel}(\be A\le)/A\lbe (K\le)^{\wedge}\ar[dr]_{\overline{\gamma}_{\be\le  A,S}}\ar[r]^(.65){\overline{\gamma}_{\be\le  A,S}^{\e\prime}}&C_{\be A\le,\e S}\ar[d]^{\sim}\\
&\cok\e\widehat{\rho}\,.
}
\end{equation}
Then the vertical map in \eqref{fip2} induces an isomorphism of finite abelian groups
\begin{equation}\label{fip3}
\cok \overline{\gamma}_{\be\le  A,S}^{\e\prime}\isoto \cok \overline{\gamma}_{\be\le  A,S}\e.
\end{equation}
Next let
\begin{equation}\label{fi}
\varphi_{\be\le  A,\e S}\colon \e T\e\sha^{1}\be(\lbe A\le)_{\rm div}\to C_{\be A\le,\le S}
\end{equation}
be the composition
\[
T\e\sha^{1}\be(\lbe A\le)_{\rm div}\underset{\eqref{iso1}}\isoto T\e{\rm Sel}(\lbe A)/A\lbe (K\le)^{\wedge}\overset{\!\overline{\gamma}_{\be\le  A,S}^{\e\prime}}{\lra}
C_{\be A\le,\e S},
\]
where $\overline{\gamma}_{\be\le  A,S}^{\e\prime}$ is the map \eqref{fip}. Then there exists a canonical isomorphism of 
finite abelian groups
\begin{equation}\label{fi2}
\cok\varphi_{\be\le  A,\e S}\isoto \cok \overline{\gamma}_{\be\le  A,S},
\end{equation}
namely the composition
\[
\cok\varphi_{\be\le  A,\e S}\isoto \cok \overline{\gamma}_{\be\le  A,S}^{\e\prime}\underset{\eqref{fip3}}\isoto \cok \overline{\gamma}_{\be\le  A,S}\,,
\]
where the first map is induced by the identity map on $C_{\be A\le,\e S}$. Finally, we define a map
\begin{equation}\label{fi3}
\cok\varphi_{\be\le  A,\le S}\isoto
\cok \gamma_{\lbe A,\le S}
\end{equation}
by the commutativity of the diagram
\[
\xymatrix{\cok\varphi_{\lbe  A,\le S}\ar[dr]_(.4){\eqref{fi2}}^{\sim}\ar[r]^(.5){\sim}&\cok \gamma_{\lbe  A,\le S}\ar[d]_(.45){\sim}^(.45){\eqref{gdi2}}\\
&\cok \overline{\gamma}_{\lbe  A,\le S}\,.
}
\]
Then the composition
\[
\cok\varphi_{\be\le  A,\le S}\underset{\eqref{fi3}}{\isoto}
\cok \gamma_{\lbe A,\le S}\underset{\eqref{can2}}{\isoto}(\e\sha^{1}_{\le\nr}\lbe(S,A^{ t}\e)/\e\sha^{1}\lbe(\lbe A^{t}))^{*}
\] 
is a canonical isomorphism of finite abelian groups. Thus we obtain

\begin{theorem}\label{main1} Let $A$ be an abelian variety over $K$ with dual abelian variety $A^{t}$. Then there exists a canonical perfect pairing of finite abelian groups
\[
C_{\lbe A\le,\e S}/\varphi_{\be\lle A\le,\le S}\lbe(\le T\e\sha^{1}\be(\lbe A\le)_{\rm div}\lbe)\times \sha^{1}_{\le\nr}\lbe(S,A^{t}\le)/\e\sha^{1}\be(\be A^{ t}\le)\to\Q/\Z,
\]
where $C_{\lbe A\le,\e S}$ is the group \eqref{ncg}, $\varphi_{\be\lle A\le,\le S}$ is the map \eqref{fi} and $\sha^{1}_{\le\nr}\lbe(S,A^{t}\le\lle)$ is the group \eqref{sha-nr}.
\end{theorem}

\section{Jacobian varieties}
Recall from the Introducion the curve $X$ over $K$. Let 
\[  
P=\pic_{X\lbe/\lbe K}
\]
be the Picard scheme of $\e X\e$ over $K$. Then $P$ is a smooth and commutative $K$-group scheme whose identity component
\[
J=\pic_{X\be/\lbe K}^{0}
\]
is an abelian variety over $K$ \cite[Proposition 3, p.~244]{blr}, called the {\it Jacobian variety of $X$}. There exists a canonical exact sequence of abelian sheaves for the \'etale topology on $K$
\begin{equation}\label{seq}
0\to J\to P\overset{\!\!\rm deg}{\to}\Z_{K}\to 0.
\end{equation}
Now recall that the {\it index} (respectively, {\it period}\,) of $X$ over $K$ is the least positive degree $\delta$ (respectively, $\delta^{\le\prime}$\e) of a divisor on $X$ (respectively,  divisor class in $(\le{\rm Pic}\e X_{K^{\rm s}}\be)^{{\rm Gal}(K^{\le\rm s}\be/K)}$). If $L$ is a field containing $K$, then the index (respectively, period) of $X_{L}=X\times_{K}L$ divides the index (respectively, period) of $X$. Further, $\delta^{\le\prime}\le|\e\delta$. See, e.g., \cite{li}. The \'etale cohomology sequence associated to \eqref{seq} induces an exact sequence of abelian groups
\begin{equation}\label{seq2}
0\to\Z/\delta^{\le\prime}\lle\Z\to H^{1}\lbe(K,J\e)\to H^{1}\lbe(K,P\le)\to 0.
\end{equation}
For every prime $v$ of $K$, we will write $\delta_{v}$ for the index of $X_{\be K_{\lbe v}}$. By \cite[Remark 1.6, p.~249]{la}, $\delta_{v}=1$ for all but finitely many primes $v$ of $K$. If $S$ is a finite set of primes of $K$ containing the archimedean primes and $v\notin S$, $\delta_{v}^{\le\rm nr}$ will denote the index of $X_{\be K_{v}^{\lbe\nr}}$. Note that $\delta_{v}|\e\delta$ and $\delta^{\le\prime}_{v}|\e\delta^{\le\prime}$ for every $v$. Further, $\delta_{v}^{\le\rm nr}\e|\e\delta_{v}$ and $\delta_{v}^{\le\prime\le\rm nr}\e|\e\delta_{v}^{\le\prime}$ for every $v\notin S$.  

\smallskip

Next let
\begin{equation}\label{dmap}
D\colon \Z/\delta^{\le\prime}\e\Z\to \prod_{\textrm{all }v}\Z/\delta_{v}^{\le\prime}\e\Z
\end{equation}
be the natural diagonal map and set
\begin{equation}\label{dp}
\Delta^{\be\prime}={\rm l.c.m.}\{\delta_{v}^{\le\prime}\}.
\end{equation}
Then $\krn D=\Delta^{\be\prime}\e\Z/\lle\delta^{\le\prime}\le\Z$ and $\cok D$ is a finite abelian group of order
\begin{equation}\label{cokd}
|\cok D|=\left(\e\prod\delta_{v}^{\le\prime}\right)\!/\Delta^{\be\prime}.
\end{equation}

An application of the snake lemma to the exact and commutative diagram
\begin{equation}\label{snl}
\xymatrix{0\ar[r]&\Z/\delta^{\le\prime}\e\Z\ar[r]\ar[d]_(.43){D}&H^{1}\lbe(K,J\le)
\ar[r]\ar[d]& H^{1}\lbe(K, P\le)\ar[d]\ar[r]& 0\\
0\ar[r]&\displaystyle{\prod_{\textrm{all }v}}\Z/\delta_{v}^{\le\prime}\e\Z\ar[r]&\displaystyle{\prod_{\textrm{all }v}}H^{1}\lbe(K_{v},J\le)
\ar[r]&\displaystyle{\prod_{\textrm{all }v}} H^{1}\lbe(K_{v}, P\le)\ar[r]& 0,\\
}
\end{equation}
whose rows are induced by \eqref{seq2}, yields an exact sequence of abelian groups
\begin{equation}\label{s1}
0\to \Delta^{\be\prime}\e\Z/\lle\delta^{\le\prime}\le\Z\to \sha^{1}\lbe(J\e)\to \sha^{1}\lbe(P\le)\to \cok D.
\end{equation}

Similarly, let
\begin{equation}\label{dnrmap}
D_{\be \nr}^{\le S}\colon \Z/\delta^{\le\prime}\e\Z\to\prod_{v\notin  S}\Z/\delta_{v}^{\le\prime\le\lle\rm nr}\e\Z\times \prod_{v\in S}\Z/\delta_{v}^{\le\prime}\e\Z
\end{equation}
be the natural diagonal map and let
\begin{equation}\label{dpp}
\Delta^{\be\prime\prime}={\rm l.c.m.}\{\delta_{v}^{\le\prime\le\lle\rm nr},v\notin S,\delta_{v}^{\le\prime}, v\in S\e\}.
\end{equation}
Then there exists a canonical exact sequence of abelian groups
\begin{equation}\label{s2}
0\to \Delta^{\be\prime\prime}\Z/\lle\delta^{\le\prime}\le\Z\to \sha^{1}_{\le\nr}\lbe(S,J\e)\to \sha^{1}_{\le\nr}\lbe(S,P\le)\to \cok D_{\be \nr}^{\le S},
\end{equation}
where the groups $\sha^{1}_{\le\nr}\lbe(S,J\e)$ and $\sha^{1}_{\le\nr}\lbe(S,P\le)$ are given by  \eqref{sha-nr}.

\smallskip

The groups $\cok D$ and $\cok D_{\be \nr}^{\le S}$ are related as follows.
\begin{lemma}\label{dd} There exists a canonical exact sequence of finite abelian groups
\[
0\to\Delta^{\be\prime\prime}\Z/\lbe\Delta^{\be\prime}\le\Z\to
\displaystyle{\prod_{v\notin S}}\e\Z/\lbe d_{\le v}\e \Z\to\cok D\to \cok D_{\be \nr}^{\le S}\to 0,
\]
where the integers $\Delta^{\be\prime\prime}, \Delta^{\be\prime}$ and $d_{\le v}$ are given by \eqref{dp}, \eqref{dpp} and \eqref{dv}, respectively, and the maps $D$ and $D_{\be \nr}^{\le S}$ are given by \eqref{dmap} and \eqref{dnrmap}, respectively.
\end{lemma}
\begin{proof} This follows by applying Lemma \ref{ker-cok} to the pair of maps
\[
\Z/\delta^{\le\prime}\e\Z\overset{D}{\to}\prod_{\textrm{all }v}\Z/\delta_{v}^{\le\prime}\e\Z\to\prod_{v\notin S}\Z/\delta_{v}^{\le\prime\le\lle\rm nr}\e\Z\times \prod_{v\in S}\Z/\delta_{v}^{\le\prime}\e\Z
\]
whose composition is the map $D_{\be \nr}^{\le S}$.
\end{proof}

\smallskip

\begin{proposition}\label{key1} Assume that the integers $\delta_{v}^{\le\prime}$ are pairwise coprime. Then there exists a canonical exact sequence of finite abelian groups
\[
0\to \Z/\lbe d\e\Z\to \sha^{1}_{\le\nr}\lbe(S,\lbe J\e)/\le\sha^{1}\lbe(J\e)\to \sha^{1}_{\le\nr}\lbe(S,\lbe P\e)/\le\sha^{1}\lbe(\lbe P\le)\to 0,
\]
where $d$ is the integer \eqref{d} and the groups $\sha^{1}_{\le\nr}\lbe(S,\lbe J\e)$ and $\sha^{1}_{\le\nr}\lbe(S,\lbe P\e)$ are given by \eqref{sha-nr}.
\end{proposition}
\begin{proof} The hypothesis shows that $\Delta^{\be\prime}$ \eqref{dp} equals $\prod\delta_{v}^{\le\prime}$, whence $\cok D=0$ by \eqref{cokd}. Now Lemma \ref{dd} shows that $\cok D_{\be \nr}^{\le S}=0$ as well. Thus there exists a canonical exact and commutative diagram of abelian groups
\[
\xymatrix{0\ar[r]&\Delta^{\be\prime}\e\Z/\lle\delta^{\le\prime}\le\Z\ar[r]\ar@{^{(}->}[d]&\sha^{1}\lbe(J\e)\ar@{^{(}->}[d]\ar[r]&\sha^{1}\lbe(P\le)\ar@{^{(}->}[d]\ar[r]& 0\\
0\ar[r]&\Delta^{\be\prime\prime}\Z/\lle\delta^{\le\prime}\le\Z\ar[r]&\sha^{1}_{\le\nr}\lbe(S,J\e)\ar[r]& \sha^{1}_{\le\nr}\lbe(S,P\le)\ar[r]&0
}
\]
whose top and bottom rows are the sequences \eqref{s1} and \eqref{s2}, respectively. An application of the snake lemma to the above diagram yields an exact sequence of finite abelian groups
\[
0\to \Delta^{\be\prime\prime}\Z/\lle\Delta^{\be\prime}\Z\to \sha^{1}_{\le\nr}\lbe(S,J\e)/\le\sha^{1}\lbe(J\e)\to \sha^{1}_{\le\nr}\lbe(S,P\le)/\le\sha^{1}\lbe(P\le)\to 0.
\]
Now, since $\cok D=0$, Lemma \ref{dd} shows that $\Delta^{\be\prime\prime}\Z/\lle\Delta^{\be\prime}\Z$ is canonically isomorphic to
$\prod_{\e v\notin S}\le\Z/\lbe d_{\le v}\le \Z$. Further, since the integers $\delta_{v}^{\le\prime}$ are pairwise coprime, so also are the integers $d_{\le v}$ \eqref{dv}. Thus the Chinese Remainder Theorem yields a canonical isomorphism
$\prod_{\e v\notin S}\le\Z/\lbe d_{\le v}\le \Z\simeq \Z/\lbe d\e\Z$, whence the proposition follows.
\end{proof}

\smallskip

Next, there exists a canonical exact sequence of abelian groups (see \cite[pp.~400-401]{ga03})
\[
0\to\pic\, X\to P(K\le)\to\br\e K\to\br\e X\to H^{1}\lbe(K,P\le)\to 0.
\]
The above sequence induces a functorial isomorphism of abelian groups
\begin{equation}\label{bra-iso}
\br_{\lbe\textrm{a}}\e  X\isoto H^{1}\lbe(K,P\le),
\end{equation}
where $\br_{\lbe\textrm{a}}\le  X$ is the group \eqref{bra} over $K$. If $L=K_{v}$ or $\knr$, where $v$ is a prime of $K$, there exists a commutative diagram of abelian groups
\begin{equation}\label{diag}
\xymatrix{\br_{\lbe\textrm{a}}\e X\ar[d]\ar[r]^(.45){\sim}&H^{1}\lbe(K,P\le)\ar[d]\\
\br_{\lbe\textrm{a}}\e X_{\be\lle L}\ar[r]^(.45){\sim}& H^{1}\lbe(L,P\le)\le,
}
\end{equation}
where the horizontal arrows are the maps \eqref{bra-iso} over $K$ and over $L\e$, the left-hand vertical arrow is induced by the pullback homomorphisms $\br\e K\to \br\e L$ and $\br\e X\to \br\, X_{\be\lle L}$ and the right-hand vertical map is the restriction map in Galois cohomology.

By the commutativity of diagram \eqref{diag}, the map \eqref{bra-iso} and its analogs over $K_{v}$ for every $v$ and over $\knr$ for every $v\notin S$ yield isomorphisms of abelian groups
\begin{eqnarray}
\bha_{\le\nr}(S, X\lbe)&\isoto&\sha^{1}_{\le\nr}\lbe(S,P\e)\label{bra=shap1}\\
\bha\lle(\lbe X)&\isoto& \sha^{1}(P\e)\label{bra=shap2},
\end{eqnarray}
where the groups $\bha_{\le\nr}(S, X\lbe)$ and $\bha(\lbe X\lbe)$ are given by \eqref{brasnr} and \eqref{bra-sha}, respectively, and the group $\sha^{1}_{\le\nr}\lbe(S,P\e)$ is defined by \eqref{sha-nr}.

\begin{proposition}\label{key2} Assume that the integers $\delta_{v}^{\le\prime}$ are pairwise coprime. Then there exists a canonical exact sequence of finite abelian groups
\[
0\to \Z/\lbe d\e\Z\to \sha^{1}_{\le\nr}\lbe(S,\lbe J\e)/\le\sha^{1}\lbe(J\e)\to \bha_{\le\nr}(S, X\lbe)/\e\bha(\lbe X)\to 0,
\]
where $d$ is the integer \eqref{d}, $\sha^{1}_{\le\nr}\lbe(S,\lbe J\e)$ is given by \eqref{sha-nr} and the groups $\bha_{\le\nr}(S, X\lbe)$ and 
$\bha(\lbe X)$ are given by \eqref{brasnr} and \eqref{bra-sha}, respectively.
\end{proposition}
\begin{proof} This is immediate from Proposition \ref{key1} using the isomorphisms \eqref{bra=shap1} and \eqref{bra=shap2}.
\end{proof}

Via the autoduality of the Jacobian \cite[Theorem 6.6]{mij}, the preceding proposition and Theorem \ref{main1} yield the second main result of this paper.

\begin{theorem}\label{main2}  Assume that the integers $\delta_{v}^{\le\prime}$ are pairwise coprime. Then there exists a canonical exact sequence of finite abelian groups
\[
0\to\Hom(\e\bha_{\le\nr}(S, X)\lbe/\le\bha(\lbe X),\Q/\Z\e)\to C_{\lbe J,\e S}/\varphi_{\lbe J,\e S}(\le T\e\sha^{1}\be(\be J\le)_{\rm div}\lbe)\to \Z/\lbe d\e\Z\to 0,
\]
where $C_{\be J\le,\e S}$ is the N\'eron $S$-class group of $J$ \eqref{ncg}, $\varphi_{\be\lle J\le,\le S}$ is the map \eqref{fi} and $d$ is the integer \eqref{d}.
\end{theorem}

Since $d=\prod_{\e v\notin S}(\delta^{\le\prime}_{v}/\delta_{v}^{\le\prime\le\lle\rm nr})=1$ when 
$\delta_{v}^{\le\prime}=1$ for every $v\notin S$, the following statement is an immediate consequence of the theorem.

\begin{corollary}\label{cor2} Let $X$ be a proper, smooth and geometrically connected curve over $K$ with associated Jacobian variety $J$. Assume that $\delta_{v}^{\le\prime}=1$ for every prime $v\notin S$. Then there exists a canonical perfect pairing of finite abelian groups
\begin{equation}\label{jac}
C_{\lbe J,\e S}/\varphi_{\lbe J,\e S}(\le T\e\sha^{1}\be(\be J\le)_{\rm div}\lbe)\times \bha_{\le\nr}(S, X)\lbe/\le\bha(\lbe X)\to \Q/\Z.
\end{equation}
\end{corollary}

\section{Concluding remarks}

Let $K$ be a number field and let $G$ be the generic fiber of a reductive group scheme over $\spec\mathcal O_{K,\e S}$. By \cite[Theorem 5.5]{ga13}, there exists a canonical perfect pairing of finite abelian groups
\begin{equation}\label{cab}
C_{\rm ab}\lbe(G\le)\times {\rm Br}_{\rm a,\e nr}^{S}(G\le)/\e\bha(G\e)\to\Q/\Z,
\end{equation}
where $C_{\rm ab}\lbe(G\le)$ is the {\it abelian $S$-class group of $G$} \cite[Definition 2.8]{ga13} and the subgroup ${\rm Br}_{\rm a,\e nr}^{S}(G)$ of ${\rm Br}_{\rm a}(G\le)$ is defined in \cite[(5.4), p.~212]{ga13}. Thus, the pairing \eqref{jac} above may well be regarded as an analog of \eqref{cab} for a certain class of Jacobian varieties over $K$.

On the other hand, by \cite[Theorem 8.12, p.~65]{san}, there exists a perfect pairing of finite abelian groups
\[
A(G\e)\times \bha_{\omega}(G\e)/\e\bha(G\e)\to\Q/\Z,
\]
where $A(G\e)=(\prod G(K_{v}))/\e\overline{G(K\e)}$ is the defect of weak approximation on $G$
and $\bha_{\e\omega}(G\e)$ denotes the subgroup of ${\rm Br}_{\rm a}(G\le)$ of classes  that are locally trivial at all but finitely many primes of $K$. The preceding statement and Corollary \ref{cor2} together suggest that the group $C_{\lbe J,\e S}/\varphi_{\lbe J,\e S}(\le T\e\sha^{1}\be(\be J\le)_{\rm div}\lbe)$ is closely related to (but perhaps not identical with) the defect of weak approximation on $J$, as suggested by the referee. In a subsequent publication we will try to clarify the connections that may well exist between the N\'eron $S$-class group $C_{\lbe A,\e S}$ of an abelian variety $A$ (over any global field $K$\le), the defect of weak approximation on $A$ and the associated Brauer-Manin obstruction.

\end{document}